\documentclass[10pt,amsfonts, epsfig]{amsart}
\usepackage{amsmath, amscd, amssymb}
\usepackage{graphpap, color}
\usepackage[mathscr]{eucal}
\usepackage{mathrsfs}
\usepackage{pstricks}
\usepackage{cancel}
\usepackage[mathscr]{eucal}
\usepackage{verbatim}
\usepackage[all]{xy}
\usepackage{stmaryrd}
\usepackage[numbers,sort&compress]{natbib}

\def\bB{{\mathbf B}}

\def\cA{{\cal A}}

\numberwithin{equation}{section}

\newcommand{\CC}{\mathbb{C}}

\newcommand{\cal}{\mathcal}

\def\cC{{\cal C}}
\def\cD{{\cal D}}

\def\cO{{\cal O}}

\def\cT{{\cal T}}

\def\fB{\mathfrak{B}}

\def\mapright#1{\,\smash{\mathop{\lra}\limits^{#1}}\,}

\def\sta{^\ast}

\def\sta{^{\ast}}

\def\sta{^*}

\def\lra{\longrightarrow}

\def\Oplus{\mathop{\oplus}}

\newcommand{\coh}{\mathrm{H}}

\def\begeq{\begin{equation}}
\def\endeq{\end{equation}}
\def\and{\quad{\rm and}\quad}

\def\sub{\subset}

\def\and{\quad\text{and}\quad}

\DeclareMathOperator{\image}{Im}

 \DeclareMathOperator{\rank}{rank}

\DeclareMathOperator{\Res}{Res}
\DeclareMathOperator{\Ho}{H}
\DeclareMathOperator{\kernel}{Ker}

\newtheorem{prop}{Proposition}[section]

\newtheorem{theo}[prop]{Theorem}
\newtheorem{lemm}[prop]{Lemma}
\newtheorem{coro}[prop]{Corollary}
\newtheorem{rema}[prop]{Remark}

\newtheorem{defi-prop}[prop]{Definition-Proposition}

\def\dbar{\overline{\partial}}

\def\sta{^\ast}
\def\image{\text{Im}\,}

\def\sD{{\mathscr D}}

\def\beq{\begin{equation}}
\def\eeq{\end{equation}}

\def\bee{\begin{equation}}
\def\eeq{\end{equation}}

\title{Serre Duality for the Cohomology of Landau-Ginzburg models}

\author[Mu-Lin Li]{Mu-Lin Li}
\address{College of Mathematics and Econometrics, Hunan University, China} \email{mulin@hnu.edu.cn}

\date{}

\begin{document}
\maketitle

\begin{abstract}
  Let $V$ and $F$ be holomorphic bundles over a complex manifold $M$, and $s$ be a holomorphic section of $V$. We study the cohomology associated to the Koszul complex induced by $s$, and prove a generalized Serre duality theorem for them.
\end{abstract}
\section{Introduction}

 The Serre duality theorem is a fundamental result in complex manifold, which establishes a
duality between the cohomology of a complex manifold and the cohomology of with compact supports, provided that the $\dbar$ operator has closed range in appropriate degrees. In this paper we extend the Serre duality to the cohomology of Landau-Ginzburg models.

Let $V$ be a holomorphic bundle over a complex manifold (usually noncompact) $M$ with  $\mathrm{rank}\, V=\dim M=n$, and $s$ be a holomorphic section of $V$ with compact zero loci $Z:=(s^{-1}(0))$.  Let $V^{\vee}$ be the dual bundle of $V$,  then $s$ induced the following Koszul complex
 \beq\label{KZ}0\to \wedge^nV^{\vee}\mapright{\iota_s}\cdots \mapright{\iota_s} \wedge^2 V^{\vee}\mapright{\iota_s} V^{\vee}\mapright{\iota_s} \CC\to0,\eeq
  where $\iota_s$ is the contraction operator induced by $s$.

Let $F$ be another holomorphic bundle over $M$, we have the following complex from (\ref{KZ})

\beq\label{KZ1}0\to \wedge^nV^{\vee}\otimes F\mapright{\iota_s\otimes 1}\cdots \mapright{\iota_s\otimes 1} \wedge^2 V^{\vee}\otimes F\mapright{\iota_s\otimes 1} V^{\vee}\otimes F\mapright{\iota_s\otimes 1}  F\to0.\eeq
Denote by $\mathbb{H}^{\bullet}(M;V,F)$ the hypercohomology associated to the above complex. We will, somewhat abusively, write $\iota_s\otimes 1$ as $\iota_s$. Because (\ref{KZ1}) is exact outside the compact set $Z$, the cohomology $\mathbb{H}^{\bullet}(M;V,F)$ is finite dimensional over $\CC$. The study of this type cohomology origins to the mathematical interpretation of the Landau-Ginzburg models, which had been widely studied in the following papers \cite{BDLT1,BDLT2,ML1,DL1,HL1,L1,L2}.

Let $\psi\in\Gamma(M,\det V\otimes \det \Omega_M)$ be a holomorphic section, where $\Omega_M$ is the holomorphic cotangent bundle of $M$. There is a canonical pairing
\beq(-,-)_{\psi}:\mathbb{H}^{\bullet}(M;V,F)\times \mathbb{H}^{\bullet}(M;V,F^{\vee})\to \CC,\nonumber\eeq see (\ref{pair3}). Then we have the following duality theorem.

\begin{theo}\label{1th}Let $V,~F$ be holomorphic bundles over the complex manifold $M$ with  $\mathrm{rank}\, V=\dim M=n$, and $s$ be a holomorphic section of $V$ with compact zero loci $Z=s^{-1}(0)$. Assume that $\psi\in\Gamma(M,\det V\otimes \det \Omega_M)$ is nowhere vanishing. Then the above pairing $(-,-)_{\psi}$ is non-degenerate. Thus for $-n\le k\le n$,
$$\mathbb{H}^{k}(M;V,F)\cong \mathbb{H}^{-k}(M;V,F^{\vee})^{\vee}.$$
\end{theo}
This is a generalization of the non-degenerate theorem of \cite[Theorem A]{DL1} and \cite[Theorem 1.2]{LML}.

Let $V=T_M$ be the holomorphic tangent bundle of the compact complex manifold $M$, and $F$ be a holomorphic bundle over $M$. Let $s$ be the zero section of $T_M$ and $\psi=c\in \Gamma(M,\cO_M)\cong\Gamma(M,\det T_M\otimes \det \Omega_M)$ be a nonzero constant, then we recover the classical Serre duality theorem.
\begin{coro}Let $F$ be a holomorphic bundle over a compact complex manifold $M$, then
$$
\Ho^{p,q}(M,F)\cong\Ho^{n-p,n-q}(M,F^{\vee})^{\vee}.
$$
\end{coro}

Let $V=\Omega_{M}$ be the holomorphic cotangent bundle of a small open ball $M=\{z\in\CC^n||z|<\epsilon\}$, and $F=\cO_M$. Let $s=df$ be a holomorphic section of $V$, where $f$ is a holomorphic function on $M$. Let $\psi=dz_1\wedge\cdots\wedge dz_n\otimes e_1\wedge\cdots\wedge e_n$, where $\{z_i\}$ is the coordinate of $\CC^n$ and $\{e_i\}$ is the holomorphic frame of $\Omega_{M}$. Assume that $s=df=f_1e_1+\cdots+f_ne_n$ and $Z=s^{-1}(0)=0$, then using Theorem \ref{1th} and the proof of \cite[Theorem 1.3]{LML} we have
$$
\mathbb{H}^{0}(M;V,F)\cong\Gamma(M,\cO_{M})/(f_1,\cdots,f_n);\quad \quad \mathbb{H}^{k}(M;V,F)=0, \quad k\neq 0.
$$
%
%
%
%
For $g,h\in\Gamma(M,\cO_{M})$,
 let $\psi^{\prime}=gh\psi$. By (\ref{formula3}), we have
\begin{eqnarray*}
(g,h)_{\psi}
&=&(-1)^{\lfloor \frac{n+3}{2}\rfloor+\frac{n(n+1)}{2}}(-2\pi i)^n\Res\frac{\psi^{\prime}}{s},
\end{eqnarray*}
where $\Res\frac{\psi^{\prime}}{s}$ is the virtual residue associated to $\psi'$ and $s$, the symbol $\lfloor\frac{n+3}{2}\rfloor$ is the greatest integer less than or equal to $\frac{n+3}{2}$. The virtual residue, which had been constructed by Chang and the author in  \cite{ML1}, coincides with the Grothendieck residue up to a sign. Therefore $\Res\frac{\psi^{\prime}}{s}$ equals to the Grothendieck residue $\mathrm{res}_s(g,h)=\int_{|f_i|=\epsilon_i}\frac{gh dz_1\wedge\cdots\wedge dz_n}{f_1\cdots f_n}$ up to a sign, see formula (\ref{equa9}).
Thus we recover the following local duality theorem, see \cite[Page 659]{GH}.
\begin{coro}Let $V=\Omega_{M}$ be the holomorphic cotangent bundle of a small open ball $M=\{z\in\CC^n||z|<\epsilon\}$, and $F=\cO_M$. Let $s=df$ be a holomorphic section of $V$, where $f$ is a holomorphic function on $M$. Let $\psi=dz_1\wedge\cdots\wedge dz_n\otimes e_1\wedge\cdots\wedge e_n$, where $\{z_i\}$ is the coordinate of $\CC^n$ and $\{e_i\}$ is the holomorphic frame of $\Omega_{M}$. Assume that $s=df=f_1e_1+\cdots+f_ne_n$ and $Z=s^{-1}(0)=0$. Then
$$\mathrm{res}_s:\Gamma(M,\cO_{M})/(f_1,\cdots,f_n)\times \Gamma(M,\cO_{M})/(f_1,\cdots,f_n)\to \CC$$ is non-degenerate.
\end{coro}

\noindent{\bf Acknowledgment}:  The author thanks C. I. Lazaroiu and Wanmin Liu for inspired discussion. He also thanks the IBS Center for Geometry and Physics for hospitality during his visit on May, 2018. This work was supported by the Start-up Fund of Hunan University.

\section{Cohomology with compact support}

In this section, we study different types of hypercohomology associate to the exact sequence (\ref{KZ1}). It is similar to Section 2 in \cite{LML}. As before let $V,~F$ be holomorphic bundles over a complex manifold $M$ with  $\text{rank}\, V=\dim M=n$, and $s$ is a holomorphic section of $V$ with compact zero loci $Z=s^{-1}(0)$.  Let $V^{\vee}$ be the dual bundle of $V$.

Let $\cA^{i,j}(\wedge^l V^{\vee}\otimes F)$ be the sheaf of smooth $(i,j)$ forms
 on $M$ with value in $\wedge^lV^{\vee}\otimes F$. Let $\Omega^{(i,j)}(\wedge^l V^{\vee}\otimes F):=\Gamma(M,\cA^{i,j}(\wedge^l V^{\vee}\otimes F))$
and assign its element $\alpha$ to have degree $\sharp \alpha=i+j-l$. Let $$\Omega_c^{(i,j)}(\wedge^l V^{\vee}\otimes F):=\{\alpha|\alpha\in\Gamma(M,\cA^{i,j}(\wedge^l V^{\vee}\otimes F))\quad\text{with compact support}\}.$$ Then
  $$\bB:=\oplus_{i,j,l}\Omega^{(i,j)}(\wedge^l V^{\vee}\otimes F)$$
is a graded commutative  algebra with the  (wedge) product uniquely extending wedge products in $\Omega^{\bullet}, \wedge^{\bullet} V^{\vee}$ and mutual tensor products. Denote
$$
\cC^k_M=\Oplus_{i-j=k}C^{i,j}_M\subset \bB \ \ \ \ \ \
\mbox{with}\quad C^{i,j}_M:= \Omega^{(0,i)}(\wedge^j V^{\vee}\otimes F)=\Gamma(M,\cA^{0,i}(\wedge^j V^{\vee}\otimes F)),
$$
 and
$$
\cC^k_{c,M}=\Oplus_{i-j=k}C^{i,j}_{c,M} \ \ \ \ \ \
\mbox{with}\quad C^{i,j}_{c,M}:= \{\alpha\in C^{i,j}_{M}| \ \alpha\   \text{has compact support}\}.
$$
 Let $\cC_M:=\oplus_k\cC^k_M$ and $\cC_{c,M}:=\oplus_k\cC^k_{c,M}$. For $\alpha\in \cC_M$ we denote $\alpha_{i,j}$ to be its component in $C^{i,j}_M$. Clearly, $\cC_{M}$ is a bi-graded  $C^{\infty}(M)$-module. Under the operations
 $$\dbar: C^{i,j}_M\lra C^{i+1,j}_M \and \iota_s : C^{i,j}_M\lra C^{i,j-1}_M$$
  the space  $C^{\bullet,\bullet}_{M}$ becomes a double  complex and $C^{\bullet,\bullet}_{c,M}$ is a subcomplex. 
  We shall study the cohomology of $\cC^{\bullet}_{M}$ and $\cC^{\bullet}_{c,M}$ with respect to the following  coboundary operator
  $$\dbar_s:=\dbar+\iota_s.$$
  One checks $\dbar_s^2=0$ using Leibniz rule of $\dbar$ and $\dbar s=0$. Denote by
  $$
  \mathbb{H}^{k}(M;V,F)=\Ho^k(\cC^{\bullet}_M),
  $$
and
 $$
  \mathbb{H}^{k}_c(M;V,F)=\Ho^k(\cC^{\bullet}_{c,M}).
$$

Fix a Hermitian metric $h_V$ on $V$.   For a nonzero $s$  on $U:=M\setminus Z$, we can form the following smooth section
 $$\bar{s}:=\frac{(*,s)_{h_V}}{(s,s)_{h_V}}\in\Gamma(U,\cA^{0,0}( V^{\vee}\otimes F)).$$ It associates a  map    $$\bar{s}\wedge:\Gamma(U,\cA^{0,i}(\wedge^j V^{\vee}\otimes F))\rightarrow\Gamma(U,\cA^{0,i}(\wedge^{j+1} V^{\vee}\otimes F)).$$
 To distinguish it in later calculation, we denote
$
 \cT_s:= \bar{s}\wedge:   C^{\bullet,\bullet}_{U}\longrightarrow C^{\bullet, \bullet+1}_{U}
$
, where $ C^{\bullet,\bullet}_{U}:=\Gamma(U,\cA^{0,\bullet}( \wedge^{\bullet}V^{\vee}\otimes F))$.

The injection $j:U\to M$ induces the restriction  $j^*:C^{\bullet, \bullet}_M\to C^{\bullet, \bullet}_U$.
Let $\rho$ be a smooth cut-off function on $M$ such that $\rho|_{U_1}\equiv 1$ and $\rho|_{M\setminus U_2}\equiv 0$
 for some  relatively compact open neighborhoods $U_1\subset \overline{U}_1 \subset U_2$ of $Z$ in $M$.

 We define the degree of an operator to be its change on the total degree of elements in $\cC_M(\cC_U)$.
Then   $\dbar$ and $\cT_s$ are of degree $1$ and $-1$ respectively, and
 $[\dbar, \cT_s]=\dbar \cT_s+ \cT_s\dbar$ is of degree $0$.  Consider two  operators  introduced   in \cite[(3.1), (3.2)]{ML1} or \cite[page 11]{LLS}
 \beq\label{defi-t1}
 T_\rho: \cC_M\to \cC_{c,M}  \qquad   \qquad T_\rho(\alpha):=\rho \alpha+(\dbar\rho)\cT_s {1\over 1+[\dbar, \cT_s]}(j^*\alpha) \eeq
and
\beq\label{defi-t2}
\ \   R_\rho: \cC_M\to \cC_M \qquad  \
 \qquad R_\rho(\alpha):= (1-\rho) \cT_s {1\over 1+[\dbar, \cT_s]}(j^*\alpha).
\eeq
Here as an operator
$$
{1\over 1+[\dbar, \cT_s]}:=\sum\limits_{k=0}^\infty(-1)^k[\dbar, \cT_s]^k
$$
is well-defined since $[\dbar, \cT_s]^k(\alpha)=0$ whenever $k>n$. Clearly $T_\rho$ is of  degree zero and $R_\rho$ is of  degree by $-1$. Also $R_\rho(\cC_{c,M})\sub \cC_{c,M}$ by definition.

\begin{lemm}\label{lemmaforquasiiso}
  {\upshape $[\dbar_s,  R_\rho]=1-T_\rho$  } as operators on $\cC_M$.
\end{lemm}
\begin{proof}
It is direct to check that
\footnote{
As a notation convention, we always denote $[,]$ for the graded commutator, that is for operators $A, B$ of degree $|A|$ and $|B|$, the bracket is given by
$$
[A,B]=AB-(-1)^{|A||B|}BA.
$$  \black}
  \beq\label{commutator1} [\iota_s,  \cT_s]=1\quad \text{on}\  \cC_U.
  \eeq
Moreover,
       $$[P, [\dbar, \cT_s]]=0$$
 for   $P$ being   $\iota_s, \dbar$ or $\cT_s$. Therefore, we have
\begin{eqnarray*}
[\dbar_s,  R_\rho]&=&[\dbar_s, 1-\rho]\cT_s {1\over 1+[\dbar, \cT_s ]}j^*+(1-\rho)[\dbar_s,  \cT_s ] {1\over 1+[\dbar, \cT_s ]}j^* \\
&=&-(\dbar\rho)\cT_s \frac{1}{ 1+[\dbar, \cT_s ]}j^*+(1-\rho)j^*\\
&=&-(\dbar\rho)\cT_s \frac{1}{ 1+[\dbar, \cT_s ]}j^*+(1-\rho)=1-T_\rho.
\end{eqnarray*}
\end{proof}
\begin{prop}\label{quasi-isom}
The embedding $(\cC_{c,M}, \dbar_s)\rightarrow (\cC_M,\dbar_s)$ is a quasi-isomorphism. Thus for $-n\le k\le n$,
$$\mathbb{H}^{k}_c(M;V,F)\cong \mathbb{H}^{k}(M;V,F).$$
\end{prop}

\begin{proof}
 By Lemma  \ref{lemmaforquasiiso} $
  \Ho\sta(\cC_{M}/\cC_{c,M}, \dbar_s)\equiv0,$ and thus the proposition follows.
\end{proof}

\section{Non-dengeneracy}

First we recall the definition of operators on wedge products of vector bundles over $M$, see \cite[Appendix]{ML1}. Let $\Omega^{(i,j)}(\wedge^k V \otimes\wedge^l V^{\vee}):=\Gamma(M,\cA^{i,j}(\wedge^k V \otimes\wedge^l V^{\vee}))$ be the smooth differential forms valued in $\wedge^k V \otimes\wedge^l V^{\vee}$. Let
$$\fB:=\oplus_{i,j,k,l}\Omega^{(i,j)}(\wedge^k V \otimes\wedge^l V^{\vee})$$ be a graded commutative algebra  extending the wedge products of $\Omega^{\bullet}, \wedge^\bullet V$ and $\wedge^\bullet V^{\vee}$. 
 The degree of $\alpha\in \Omega^{(i,j)}(\wedge^k V \otimes\wedge^l V^{\vee})$ is   $\sharp\alpha:=i+j+k-l$.  We briefly denote $A^0(\wedge^k V \otimes\wedge^l V^{\vee})=\Omega^{(0,0)}(\wedge^k V \otimes \wedge^l V^{\vee})$.\black

Set  $\kappa:\fB\to \Omega^\bullet$
which sends $\omega (e\otimes e')$(for $\omega\in\Omega^{(i,j)}, e\in \wedge^k V, e'\in \wedge^l V^{\vee})$ to $\omega \langle e,e'\rangle$, where $\langle \cdot, \cdot \rangle$
is the dual pairing between $\wedge^k V$ and $\wedge^k V^{\vee}$, and $\langle e,e'\rangle =0$ when $k\neq l$. We further extend the pairing by setting  $\langle\alpha,\beta\rangle :=\kappa(\alpha \beta)$  for $\alpha,\beta\in \fB$.
It is direct to verify \beq\label{kob}\dbar\langle\alpha,\beta\rangle =\langle\dbar\alpha,\beta\rangle +(-1)^{\sharp \alpha}\langle\alpha,\dbar\beta\rangle .\eeq

 We now define different types of contraction maps.
Given   $u\in \Omega^{(i,j)}(\wedge^k V)$ and
 $k\geq l$,  we   define
 \beq\label{operator1} u \lrcorner: \Omega^{(p,q)}(\wedge^{l}V^{\vee})\lra \Omega^{(p+i,q+j)}(\wedge^{k-l}V) \eeq
where for $\theta\in \Omega^{(p,q)}(\wedge^{l}V^{\vee})$, the $u\lrcorner \theta$ is determined by
$$\langle u\lrcorner \theta,\nu\sta\rangle =(-1)^{(i+j)l+(p+q)\sharp u+\frac{l(l-1)}{2}}\langle u,\theta\wedge \nu^*\rangle ,\qquad \forall \nu^*\in A^0(\wedge^{k-l}V^{\vee}).$$

 Given $\alpha\in A^0(V)$, we define
\beq\label{operator2}
\iota_{\alpha} : \Omega^{(i,j)}(\wedge^{k}V^{\vee}) \lra \Omega^{(i,j)}(\wedge^{k-1}V^{\vee})\eeq
where for $w\in  \Omega^{(i,j)}(\wedge^{k}V^{\vee})$, the $\iota_{\alpha}(w)$ is determined by
 \beq\label{us1}\langle \nu,\iota_{\alpha}(w)\rangle =\langle \alpha\wedge \nu,w\rangle ,\qquad \forall \nu \in A^0(\wedge^{k-1}V).\eeq

 For above $\alpha$, $\theta$ and $w$ one has
 $\iota_\alpha(w\wedge \theta)=\iota_\alpha(w)\wedge\theta +(-1)^{\sharp w}w\wedge \iota_\alpha(\theta).$

  Given $\gamma\in A^0(V^{\vee})$, we define
\beq\label{operator3}
\iota_{\gamma} : \Omega^{(i,j)}(\wedge^{k}V) \lra \Omega^{(i,j)}(\wedge^{k-1}V)\eeq
where for $\nu\in  \Omega^{(i,j)}(\wedge^{k}V)$, the $\iota_{\gamma}(\nu)$ is determined by
 \beq\label{us2}\langle \iota_{\gamma}(\nu),w\rangle =\langle  \nu,\gamma\wedge w\rangle ,\qquad \forall w \in A^0(\wedge^{k-1}V^{\vee}).\eeq

 We have the following identities.

\begin{lemm}[\cite{ML1}]\label{sign}
Given $u\in \Omega^{(i,j)}(\wedge^n V)$,  and $\theta,~\alpha,~\gamma$ as above, one has
$$
\alpha \wedge(u\lrcorner \theta)=u\lrcorner(\iota_\alpha (\theta) ),\quad\quad \iota_{\gamma}(u\lrcorner \theta)=u\lrcorner(\gamma\wedge\theta).
$$

\end{lemm}

\begin{lemm}[\cite{ML1}]\label{sign1}
For $u\in\Omega^{(i,j)} (\wedge^k V)$, $\theta\in \Omega^{(p,q)} (\wedge^l V^{\vee})$, $k\ge l$ and a smooth form $\alpha\in\Omega^{(a,b)}(M)$, we have
 $$\alpha \wedge (u\lrcorner  \theta)= u\lrcorner (\alpha \theta)\and   \dbar (u\lrcorner \theta)=(-1)^{\sharp \theta}(\dbar u)\lrcorner \theta + u \lrcorner (\dbar \theta).$$
\end{lemm}


Denote by $\sD^{p,q}(\wedge^l V\otimes F^{\vee})$ the space of $\wedge^l V\otimes F^{\vee}$-valued $(p,q)$-current, which is the dual of the space $\Omega^{n-p,n-q}_c(\wedge^l V^{\vee}\otimes F)$. There is a naturally pairing
 \beq\label{pair9}
 (-,-)_N:\sD^{n,n-i}(\wedge^j V\otimes F^{\vee})\times \Omega^{0,i}_c(\wedge^j V^{\vee}\otimes F)\to \CC,
 \eeq
 where
 $$
 (\alpha,\beta)_N:=\int_M\langle\alpha,\beta\rangle.
 $$

Denote
$$
\cD^k_M=\Oplus_{i+j-n=k}D^{i,j}_M \ \ \ \ \ \
\mbox{with}\quad D^{i,j}_M:= \sD^{(n,i)}(\wedge^jV\otimes F^{\vee}).
$$
The coboundary map $\delta_s$ of $\cD^{\bullet}_M$  is defined as follows
 \beq
 \label{express}\delta_s \alpha=\dbar\alpha+(-1)^{l+1}s\wedge\alpha,\quad\quad\quad\quad \text{for}\quad\quad \alpha\in\sD^{p,q}(\wedge^l V\otimes F^{\vee}). \eeq

By (\ref{kob}) and (\ref{us1}),
 \beq(\alpha,\dbar_s\beta)_N+(-1)^{\sharp\alpha}(\delta_s\alpha,\beta)_N=0.
 \eeq
 Thus $(\cD^{\bullet}_M, (-1)^{\bullet+1}\delta_s)$ is the dual complex of $(\cC_{c,M}^{\bullet},\dbar_s)$.
 Let
 \beq\label{def-co}
 \mathcal{H}^{k}(M;V,F^{\vee}):=\frac{\kernel (\delta_s:\cD^{k}_M\to \cD^{k+1}_M)}{\image(\delta_s:\cD^{k-1}_M\to\cD^{k}_M)}.
 \eeq

Because the complex $(\cD^{\bullet}_M, (-1)^{\bullet+1}\delta_s)$ is quasi-isomorphic to the complex $(\cD^{\bullet}_M, \delta_s)$, the pairing (\ref{pair9}) induces a pairing on the hypercohomologies which we denote by
\beq\label{pair1}
 (-,-)_N:\mathcal{H}^{\bullet}(M;V,F^{\vee})\times \mathbb{H}_c^{\bullet}(M;V,F)\to \CC.
 \eeq
\begin{theo}\label{main2}Let $V,~F$ be holomorphic bundles over the complex manifold $M$ with  $\mathrm{rank}\, V=\dim M=n$ and $s$ be a holomorphic section of $V$ with compact zero loci $Z=s^{-1}(0)$. Then the above pairing (\ref{pair1}) is non-degenerate. Thus for $-n\le k\le n$
$$
\mathbb{H}_c^{k}(M;V,F)\cong\mathcal{H}^{-k}(M;V,F^{\vee})^{\vee}.
$$
\end{theo}
\begin{proof}
Because $\mathbb{H}_c^{k}(M;V,F)$ is finite dimensional and the complex $(\cD^{\bullet}_M, (-1)^{\bullet+1}\delta_s)$ is the dual complex of $(\cC_{c,M}^{\bullet},\dbar_s)$, by applying \cite[Theorem 1.6]{LL} and \cite[Corollary 1.7]{LL}, we obtain the theorem.
\end{proof}

\begin{rema}
By the Dolbeault-Grothendieck Lemma \cite[3.29]{Dem} for current, $\mathcal{H}^{\bullet}(M;V,F^{\vee})$ (up to a shift on degree) is the hypercohomology of the following complex,
 \beq\label{KZ5}0\to\det \Omega_M\otimes F^{\vee}\mapright{-s} \det \Omega_M\otimes V\otimes F^{\vee} \mapright{(-1)^2s\wedge}\cdots \mapright{(-1)^{n}s\wedge} \det \Omega_M\otimes \wedge^nV\otimes F^{\vee}\to0,\eeq
 which is quasi-isomorphic to the complex
\beq\label{KZ9}0\to\det \Omega_M\otimes F^{\vee}\mapright{s} \det \Omega_M\otimes V\otimes F^{\vee} \mapright{s\wedge}\cdots \mapright{s\wedge} \det \Omega_M\otimes \wedge^nV\otimes F^{\vee}\to0.\eeq
\end{rema}

Let $(\overline{\cD}^{\bullet}_M,\widetilde{\delta}_s)$ be the complex with
$$
\overline{\cD}^k_M=\Oplus_{i-j=k}\overline{D}^{i,j}_M \ \ \ \ \ \
\mbox{where}\quad \overline{D}^{i,j}_M:= \sD^{(0,i)}(\wedge^jV^{\vee}\otimes F^{\vee}),
$$
and the coboundary map $\widetilde{\delta}_s$ is defined as follows
\beq
\widetilde{\delta}_s\alpha=\dbar\alpha+(-1)^{n-l+1}\iota_s \alpha,\quad\quad\quad\quad \text{for}\quad\quad \alpha\in\sD^{p,q}(\wedge^lV^{\vee}\otimes F^{\vee}).
\eeq

By the Dolbeault-Grothendieck Lemma \cite[3.29]{Dem} for current, the complex $(\overline{\cD}^{\bullet}_M,\widetilde{\delta}_s)$ is the Dolbeault resolution of the following complex

\beq\label{KZ6}0\to \wedge^nV^{\vee}\otimes F^{\vee}\mapright{-\iota_s}\cdots \mapright{(-1)^n\iota_s} V^{\vee}\otimes F^{\vee}\mapright{(-1)^{n}\iota_s}  F^{\vee}\to0.\eeq

 Denote by
 \beq
 \Ho^{k}(\overline{\cD}^{\bullet}_M):=\frac{\kernel (\widetilde{\delta}_s:\overline{\cD}^{k}_M\to \overline{\cD}^{k+1}_M)}{\image(\widetilde{\delta}_s:\overline{\cD}^{k-1}_M\to\overline{\cD}^{k}_M)}
 \eeq
the cohomology of the complex $(\overline{\cD}^{\bullet}_M,\widetilde{\delta}_s)$.

Let $\psi\in\Gamma(M,\det V\otimes \det \Omega_M)$ be a holomorphic section which is nowhere vanishing. It induces a bundle isomorphism
$$
\psi\lrcorner:\wedge^l V^{\vee}\otimes F^{\vee}\to \det \Omega_M\otimes \wedge^{n-l} V\otimes F^{\vee}.
$$
 Thus
\beq
\sD^{0,q}(\wedge^l V^{\vee}\otimes F^{\vee})\cong \sD^{n,q}(\wedge^{n-l} V\otimes F^{\vee}).
\eeq

By Lemma \ref{sign}, Lemma \ref{sign1} and (\ref{express}), the  map
$$\psi\lrcorner:\overline{\cD}^{\bullet}_M\to \cD^{\bullet}_M$$
is a complex isomorphism. Therefore for $-n\le k\le n$
\beq
 \Ho^{k}(\overline{\cD}^{\bullet}_M)\cong \mathcal{H}^{k}(M;V,F^{\vee}).
\eeq


The complex (\ref{KZ6}) is quasi-isomorphic to
\beq\label{KZ7}0\to \wedge^nV^{\vee}\otimes F^{\vee}\mapright{\iota_s}\cdots \mapright{\iota_s} V^{\vee}\otimes F^{\vee}\mapright{\iota_s}  F^{\vee}\to0, \eeq and we denote this quasi-isomorphism by $\Phi$. It induces the following hypercohomology isomorphism \beq H^{\bullet}(\Phi):\mathbb{H}^{\bullet}(M;V,F^{\vee})\cong\Ho^{\bullet}(\overline{\cD}^{\bullet}_M).\eeq
 Thus we have
\beq\label{iso-quasi}
 \mathbb{H}^{\bullet}(M;V,F^{\vee})\cong \mathcal{H}^{\bullet}(M;V,F^{\vee}).
\eeq

Combining (\ref{pair1}) and (\ref{iso-quasi}), there is a pairing
\beq\label{pair2}
(-,-)'_{\psi}:\mathbb{H}^{\bullet}(M;V,F^{\vee})\times \mathbb{H}_c^{\bullet}(M;V,F)\to \CC,
\eeq
where $(\alpha,\beta)'_{\psi}:=(\psi\lrcorner(H^{\bullet}(\Phi)(\alpha)),\beta)_N$.
%

From (\ref{iso-quasi}) and Theorem \ref{main2}, we obtain the following statement.
\begin{coro}\label{coro1}Let $V,~F$ be holomorphic bundles over the complex manifold $M$ with  $\mathrm{rank}\, V=\dim M=n$ and $s$ be a holomorphic section of $V$ with compact zero loci $Z=s^{-1}(0)$.  Assume that $\psi\in\Gamma(M,\det V\otimes \det \Omega_M)$ is a holomorphic section which is nowhere vanishing. Then the pairing $(-,-)'_{\psi}$ is non-degenerate. Thus for $-n\le k\le n$
\beq
\mathbb{H}_c^{k}(M;V,F)\cong\mathbb{H}^{-k}(M;V,F^{\vee})^{\vee}.
\eeq
\end{coro}
We define the following pairing
\beq\label{pair3}
(-,-)_{\psi}:\mathbb{H}^{\bullet}(M;V,F^{\vee})\times \mathbb{H}^{\bullet}(M;V,F)\to \CC,
\eeq
such that $(\alpha,\beta)_{\psi}:=([\alpha],[\beta])'_{\psi}$, where $[\alpha]$ and $[\beta]$ are the image of $\alpha, \beta$ under the isomorphic $\mathbb{H}^{\bullet}(M;V,F)\cong \mathbb{H}_c^{\bullet}(M;V,F)$ and $ \mathbb{H}^{\bullet}(M;V,F^{\vee})\cong \mathbb{H}^{\bullet}_c(M;V,F^{\vee})$. The pairing (\ref{pair3}) is well defined because it does not dependent on the compact support representation. Therefore
\begin{coro}\label{useful}Under the conditions as in Corollary \ref{coro1}, the pairing $(-,-)_{\psi}$ is non-degenerate. Thus for $-n\le k\le n$
\beq
\mathbb{H}^{k}(M;V,F)\cong\mathbb{H}^{-k}(M;V,F^{\vee})^{\vee}.
\eeq
\end{coro}

Let $V=\Omega_{M}$ be the holomorphic cotangent bundle of a small open ball $M=\{z\in\CC^n||z|<\epsilon\}$, and $F=\cO_M$. Let $s=df$ be a holomorphic section of $V$, where $f$ is a holomorphic function on $M$. Let $\psi=dz_1\wedge\cdots\wedge dz_n\otimes e_1\wedge\cdots\wedge e_n$, where $\{z_i\}$ is the coordinate of $\CC^n$ and $\{e_i\}$ is the holomorphic frame of $\Omega_{M}$. Assume that $s=df=f_1e_1+\cdots+f_ne_n$ and $Z=s^{-1}(0)=0$, then using Corollary \ref{useful} and the proof of \cite[Theorem 1.3]{LML}
$$
\mathbb{H}^{0}(M;V,F)\cong\Gamma(M,\cO_{M})/(f_1,\cdots,f_n);\quad \quad \mathbb{H}^{k}(M;V,F)=0, \quad k\neq 0.
$$
On the other hand $s$ induced the following complex
\beq\label{KZ4}0\to\det \Omega_M\mapright{-s} \det \Omega_M\otimes V \mapright{(-1)^2s\wedge}\cdots \mapright{(-1)^n s\wedge} \det \Omega_M\otimes \wedge^nV\to0.\eeq
%

 As in the definition of (\ref{def-co}), let $\mathcal{H}^{\bullet}(M;V,\cO_M)$ be the hypercohomology of (\ref{KZ4}), and $\mathcal{H}_c^{\bullet}(M;V,\cO_M)$ be its hypercohomology with compact support. By \cite[Proposition 3.2]{ML1}
\beq\label{equa1}
\mathcal{H}^{k}(M;V,\cO_M)\cong \mathcal{H}_c^{k}(M;V,\cO_M),
\eeq
for $-n\le k\le n$.

Because $\psi=dz_1\wedge\cdots\wedge dz_n\otimes e_1\wedge\cdots\wedge e_n$  is nowhere vanishing. It induces a bundle isomorphism
$$
\psi\lrcorner:\wedge^l V^{\vee}\to \det \Omega_M\otimes \wedge^{n-l} V.
$$
So is the Dolbeault resolution of the two complex (\ref{KZ}) and (\ref{KZ4}).

For $g,~h\in\Gamma(M,\cO_{M})$, let $[g],~[h]$ be the images of $g,~h$ under the isomorphism $\mathbb{H}^{0}(M;V,F)\cong \mathbb{H}_c^{0}(M;V,F)$. By (\ref{equa1}), (\ref{pair2}) and (\ref{pair3}) the pairing
\begin{eqnarray*}
(g,h)_{\psi}&=&\int_M\langle(-1)^{\lfloor \frac{n+3}{2}\rfloor}g\psi,[h]\rangle\\
&=&(-1)^{\lfloor \frac{n+3}{2}\rfloor}\int_M \langle gh\psi,[1]\rangle\\
&=&(-1)^{\lfloor \frac{n+3}{2}\rfloor+\frac{n(n+1)}{2}}\int_M (gh\psi)\lrcorner[1].
\end{eqnarray*}
 Let $\rho$ be a smooth cut-off function on $M$ such that $\rho|_{U_1}\equiv 1$ and $\rho|_{M\setminus U_2}\equiv 0$
 for some  relatively compact open neighborhoods $U_1\subset \overline{U}_1 \subset U_2$ of $0$ in $M$. By Lemma \ref{lemmaforquasiiso}, we have
\beq
[\dbar_s, R_\rho]=1-T_\rho.\nonumber
\eeq
Therefore
\beq
\int_M (gh\psi)\lrcorner[1]=\int_M (gh\psi)\lrcorner[T_\rho 1].
\eeq
For the nonzero $s$  on $U:=M\setminus Z$, the smooth section $\bar{s}:=\frac{(*,s)_{h_V}}{(s,s)_{h_V}}\in\Gamma(U,\cA^{0,0}( V^{\vee}))$ induces a  contraction    $$\iota_{\bar{s}}:\Gamma(U,\cA^{0,i}(\wedge^j V))\rightarrow\Gamma(U,\cA^{0,i}(\wedge^{j-1} V)).$$

Denote
$$
\widetilde{\cC}^k_M=\Oplus_{i+j=k}\widetilde{C}^{i,j}_M \ \ \ \ \ \
\mbox{with}\quad \widetilde{C}^{i,j}_M:= \Omega^{(0,i)}(\wedge^j V)=\Gamma(M,\cA^{0,i}(\wedge^j V)),
$$
 and
$$
\widetilde{\cC}^k_{c,M}=\Oplus_{i+j=k}\widetilde{C}^{i,j}_{c,M} \ \ \ \ \ \
\mbox{with}\quad \widetilde{C}^{i,j}_{c,M}:= \{\alpha\in \widetilde{C}^{i,j}_{M}| \ \alpha\   \text{has compact support}\}.
$$
Let $\widetilde{\cC}_M:=\oplus_k\widetilde{\cC}^k_M$ and $\widetilde{\cC}_{c,M}:=\oplus_k\widetilde{\cC}^k_{c,M}$. Let $j:U\to M$ be the injection. We can form the following operator by using the contraction $\iota_{\bar{s}}$
\beq
 \widetilde{T}_\rho: \widetilde{\cC}_M\to \widetilde{\cC}_{c,M}  \qquad   \qquad \widetilde{T}_\rho(\alpha):=\rho \alpha+(\dbar\rho)\iota_{\bar{s}} {1\over 1+[\dbar, \iota_{\bar{s}}]}(j^*\alpha). \eeq

Because $\psi$ is holomorphic, by Lemma \ref{sign}
\beq
\int_M (gh\psi)\lrcorner[T_\rho 1]=\int_M \widetilde{T}_\rho(gh\psi).
\eeq

Denote $\psi^{\prime}=gh\psi$, and applying \cite[Proposition 3.3]{ML1} to $\psi^{\prime}$ and $s=df$, we have
\begin{eqnarray}\label{formula3}
(g,h)_{\psi}&=&(-1)^{\lfloor \frac{n+3}{2}\rfloor+\frac{n(n+1)}{2}}\int_M (gh\psi)\lrcorner[T_\rho 1]\\
&=&(-1)^{\lfloor \frac{n+3}{2}\rfloor+\frac{n(n+1)}{2}}\int_M \widetilde{T}_\rho(gh\psi)\nonumber\\
&=&(-1)^{\lfloor \frac{n+3}{2}\rfloor+\frac{n(n+1)}{2}}(-2\pi i)^n\Res\frac{\psi^{\prime}}{s},\nonumber
\end{eqnarray}
where $\Res\frac{\psi^{\prime}}{s}$ is the virtual residue associated to $\psi'$ and $s$, it coincides with the Grothendieck residue $\mathrm{res}_s(g,h)=\int_{|f_i|=\epsilon_i}\frac{gh dz_1\wedge\cdots\wedge dz_n}{f_1\cdots f_n}$ up to a sign, see formula (\ref{equa9}).
Thus we recover the local duality theorem, see \cite[Page 659]{GH}.
\begin{coro}Let $V=\Omega_{M}$ be the holomorphic cotangent bundle of a small open ball $M=\{z\in\CC^n||z|<\epsilon\}$, and $F=\cO_M$. Let $s=df$ be a holomorphic section of $V$, where $f$ is a holomorphic function on $M$. Let $\psi=dz_1\wedge\cdots\wedge dz_n\otimes e_1\wedge\cdots\wedge e_n$, where $\{z_i\}$ is the coordinate of $\CC^n$ and $\{e_i\}$ is the holomorphic frame of $\Omega_{M}$. Assume that  $s=df=f_1e_1+\cdots+f_ne_n$  and $Z=s^{-1}(0)=0$. Then
$$\mathrm{res}_s:\Gamma(M,\cO_{M})/(f_1,\cdots,f_n)\times \Gamma(M,\cO_{M})/(f_1,\cdots,f_n)\to \CC$$ is non-degenerate.
\end{coro}
\section{Appendix}
In this appendix we recall the construction of the virtual residue given by Chang and the author in \cite{ML1}, and prove the relation between the virtual residue and the Grothendieck residue when the zero loci is zero-dimensional.

Let $V$ be a holomorphic bundle over a compact complex manifold $M$ with $\rank V=\dim M=n$. Let $s$ be a  holomorphic section of  $V$, and  $Z=s^{-1}(0)$ be the compact zero loci.

Let $U:=M\setminus Z$, and  let $V_U$  be the restriction of $V$  over $U$.  Since $s$ is nowhere zero over $U$, the following Koszul sequence is exact over $U$
$$
    0\lra K_U\mapright{s} K_U\otimes V_U \mapright{ s\wedge} \cdots \mapright{ s\wedge}  K_U\otimes\wedge ^{n-1}V_U \mapright{ s\wedge} K_U\otimes\wedge^n V_U\lra 0.
$$
The exact Koszul sequence induces a homomorphism
  \beq\label{ONE}
  \coh^0(U,K_U\otimes\wedge^n V_U)\lra \coh^{n-1}(U,K_U).
  \eeq
  One also has a canonical Dolbeault isomorphism
 \beq\label{TWO}
 \coh^{n-1}(U,K_U) \cong \coh_{\bar{\partial}}^{n,n-1}(U).
 \eeq
Applying  \eqref{ONE} and \eqref{TWO} to the holomorphic section $\psi\in \Gamma(M,K_M\otimes \det{V})$, and using that every $(n,n-1)$ form is $\partial$-closed, one obtains a (unique) De-Rham cohomology class
   \beq\label{eta}
   \eta_\psi\in \coh^{2n-1}(U,\CC).
  \eeq
Then the virtual residue is defined as
\beq\label{V-residue}\mathrm{Res}_{Z}\frac{\psi}{s}:=\bigg(\frac{1}{2\pi \sqrt{-1}}\bigg)^n\int_{N}\eta_\psi\in\CC,\eeq
where   $N$ is a real $(2n-1)$-dimensional piecewise smooth compact subset of $M$ that surrounds $Z$, in the sense that  $N=\partial T$ for some compact domain $T\sub M$, which contains $Z$ and is homotopically equivalent to $Z$.

When $M=\{z\in\CC^n||z|<\epsilon\}$ is a small open ball, and $V=\Omega_{M}$ with the standard Hermitian metric $h_V$. Let $F=\cO_M$ and $s=df$, where $f$ is a holomorphic function on $M$. Let  $\{z_i\}$ be the coordinate of $\CC^n$, and $\{e_i\}$ be the holomorphic frame of $\Omega_{M}$. Assume that $s=df=f_1e_1+\cdots+f_ne_n$ and $Z=s^{-1}(0)=0$.  Let $\bar{s}=\langle s,s\rangle_{h_V}^{-1}\sum_{i=1}^n\bar{f}_i e_i^*$, where $e_i^*$ is the dual basis of $V^{\vee}$. Then we have the following equalities on $U$
$$
\dbar\iota_{\bar{s}}=\sum\big(\frac{\dbar \bar{f}_i}{\langle s,s\rangle_{h_V}}-\frac{\bar{f}_i\dbar\langle s,s\rangle_{h_V}}{\langle s,s\rangle_{h_V}^2}\big)\iota_{e_i^*},
$$
and
$$
(\langle s,s\rangle_{h_V}^{-1}\sum_{i=1}^n\bar{f}_i \iota_{e_i^*})(\sum\frac{\bar{f}_i\dbar\langle s,s\rangle_{h_V}}{\langle s,s\rangle_{h_V}^2}\iota_{e_i^*})=-\frac{\dbar\langle s,s\rangle_{h_V}}{\langle s,s\rangle_{h_V}}(\langle s,s\rangle_{h_V}^{-1}\sum_{i=1}^n \bar{f}_i \iota_{e_i^*})^2=0.
$$
Let $g,~h$ be holomorphic functions on $M$. Then $\psi=ghdz_1\wedge\cdots\wedge dz_n\otimes\  e_1\wedge\cdots\wedge e_n$ is a holomorphic section of $\Gamma(M,K_M\otimes\det V)$. Therefore
\begin{eqnarray*}
\eta_\psi&=&\langle s,s\rangle_{h_V}^{-1}(\sum \bar{f}_i\iota_{e_i^*})(\dbar\iota_{\bar{s}})^{n-1}\psi\\
&=&\langle s,s\rangle_{h_V}^{-1}(\sum \bar{f}_i\iota_{e_i^*})\big(\sum\frac{\dbar \bar{f}_i}{\langle s,s\rangle_{h_V}}\iota_{e_i^*}\big)^{n-1}\psi\\
&&-(n-1)\langle s,s\rangle_{h_V}^{-1}(\sum \bar{f}_i\iota_{e_i^*})\big(\sum\frac{\dbar \bar{f}_i}{\langle s,s\rangle_{h_V}}\iota_{e_i^*}\big)^{n-2}\big(\sum\frac{\bar{f}_i\dbar\langle s,s\rangle_{h_V}}{\langle s,s\rangle_{h_V}^2}\big)\iota_{e_i^*}\psi\\
&=&(\langle s,s\rangle_{h_V}^{-1}\sum \bar{f}_i\iota_{e_i^*})\big(\sum\frac{\dbar \bar{f}_i}{\langle s,s\rangle_{h_V}}\iota_{e_i^*}\big)^{n-1}\psi\\
&=&(-1)^{\frac{n(n-1)}{2}+\frac{n(n+1)}{2}}(n-1)!gh\sum_{i=1}^n (-1)^{i-1}\\
&&\frac{\bar{f}_i}{\langle s,s\rangle_{h_V}^n} \overline{\partial }\bar{f}_1\wedge\cdots\wedge\widehat{\overline{\partial}\bar{f}_i}\wedge\cdots\wedge\overline{\partial}\bar{f}_n \wedge dz_1\wedge\cdots\wedge dz_n.
\end{eqnarray*}
Let $N$ be a small sphere around $0$, the virtual residue

\begin{eqnarray}\mathrm{Res}_{Z}\frac{\psi}{s}&=&\bigg(\frac{1}{2\pi \sqrt{-1}}\bigg)^n\int_{N}\eta_\psi\\
&=&(-1)^{\frac{n(n+1)}{2}+\frac{n(n-1)}{2}}(n-1)!\bigg(\frac{1}{2\pi \sqrt{-1}}\bigg)^n\int_{N}gh\sum_{i=1}^n (-1)^{i-1}\nonumber\\
&&\frac{\bar{f}_i}{\langle s,s\rangle_{h_V}^n} \overline{\partial }\bar{f}_1\wedge\cdots\wedge\widehat{\overline{\partial}\bar{f}_i}\wedge\cdots\wedge\overline{\partial}\bar{f}_n \wedge dz_1\wedge\cdots\wedge dz_n.\nonumber
\end{eqnarray}
By Lemma in \cite[Page 651]{GH} and the definition of $\mathrm{res}_s(g,h)$ in \cite[Page 659]{GH}, we have
\beq\label{equa9}
\mathrm{Res}_{Z}\frac{\psi}{s}=(-1)^{\frac{n(n+1)}{2}}\mathrm{res}_s(g,h).
\eeq
\bibliographystyle{amsplain}

\end{document}